\documentclass[11pt]{amsart}
\usepackage[pagewise]{lineno}
\usepackage{amsmath,amssymb,latexsym,soul,cite,mathrsfs}

\usepackage{color,enumitem,graphicx}
\usepackage[colorlinks=true,urlcolor=blue,
citecolor=red,linkcolor=red,linktocpage,pdfpagelabels,
bookmarksnumbered,bookmarksopen]{hyperref}
\usepackage[english]{babel}

\usepackage[left=2.9cm,right=2.9cm,top=2.8cm,bottom=2.8cm]{geometry}
\usepackage[hyperpageref]{backref}

\usepackage[colorinlistoftodos]{todonotes}
\makeatletter
\providecommand\@dotsep{5}
\def\listtodoname{List of Todos}
\def\listoftodos{\@starttoc{tdo}\listtodoname}
\makeatother

\numberwithin{equation}{section}

\newcommand{\Om} {\Omega}
\newcommand{\la} {\lambda}

\newtheorem{Theorem}{Theorem}[section]
\newtheorem{Lemma}[Theorem]{Lemma}

\newtheorem{Remark}[Theorem]{Remark}
\newtheorem{Definition}[Theorem]{Definition}

\newcommand\R{\mathbb R}
\newcommand\N{\mathbb N}

\begin{document}

\title[Weighted singular problem]
{On the regularity and existence of weak solutions for a class of degenerate singular elliptic problem}

\author{Prashanta Garain}

\address[Prashanta Garain ]
{\newline\indent Department of Mathematical Sciences
	\newline\indent
Indian Institute of Science Education and Research Berhampur,
	\newline\indent
Berhampur, Odisha 760010, India,
\newline\indent
Email: {\tt pgarain92@gmail.com} }

\pretolerance10000

\subjclass[2020]{35J75, 35J92, 35J70, 35D30.}
\keywords{Singular elliptic problem, variable exponent, $p$-admissible weights, existence, weighted $p$-Laplace equation.}

\begin{abstract}
In this article, we consider a class of degenerate singular problems. The degeneracy is captured by the presence of a class of $p$-admissible weights, which may vanish or blow up near the origin. Further, the singularity is allowed to vary inside the domain. We provide sufficient conditions on the weight function, on the singular exponent and the source function to establish regularity and existence results.
\end{abstract}

\maketitle

\section{Introduction}
In this article, we consider the following class of degenerate singular elliptic problem
\begin{equation}\label{maineqn}
-\text{div}(A(x,\nabla u))=f(x)u^{-\gamma(x)}\text{ in }\Om,\quad u>0\text{ in }\Om,\quad u=0\text{ on }\partial\Om,
\end{equation}
where $\Om\subset\mathbb{R}^N$ is a bounded smooth domain with $N\geq 2$, $1<p<\infty$, $\gamma\in C(\overline{\Omega})$ is a positive function and $f\in L^m(\Omega)\setminus\{0\}$ is nonnegative, for some $m\geq 1$ to be made precise later on. Here $A:=A(x,\xi):\Om\times\mathbb{R}^N\to\mathbb{R}^N$ is a function, which is measurable in $x$ for every $\xi\in\mathbb{R}^N$ and continuous in $\xi$ for almost every $x\in \Om$. Further, for almost every $x\in\Om$ and every $\xi\in\mathbb{R}^N$, the function $A$ satisfy the following four hypothesis: 
\begin{enumerate}
\item[$(H_1)$] $|A(x,\xi)|\leq\beta w(x)|\xi|^{p-1}$ for some constant $\beta>0$,
\item[$(H_2)$] $A(x,\xi)\xi\geq\alpha w(x)|\xi|^p$ for some constant $\alpha>0$,
\item[$(H_3)$] $\langle A(x,\xi_1)-A(x,\xi_2),\xi_1-\xi_2\rangle>0$ for every $\xi_1,\xi_2\in\mathbb{R}^N$, $\xi_1\neq\xi_2$ and
\item[$(H_4)$] $A(x,\la\xi)=\la|\la|^{p-2}A(x,\xi)$ whenever $\la\in\mathbb{R}\setminus\{0\}$.
\end{enumerate}

We will assume that the weight function $w$ belong to a class of $p$-admissible weights, which may vanish or blow up near the origin (for example, $w(x)=|x|^\nu$,\,$\nu\in\mathbb{R}$) to be discussed below. The degeneracy is captured by such behavior of the weights.

We observe that equation \eqref{maineqn} extends the following weighted anisotropic singular problem
\begin{equation}\label{waeqn}
-\text{div}(w(x)H(\nabla u)^{p-1}\nabla H(\nabla u))=f(x)u^{-\gamma(x)}\text{ in }\Om,\quad u>0\text{ in }\Om,\quad u=0\text{ on }\partial\Om,
\end{equation}
which can be formed by setting
\begin{equation}\label{ex}
A(x,\xi)=\frac{1}{p}w(x)\nabla F(\xi),
\end{equation} 
in \eqref{maineqn}, see \cite{Juh}. Here $1<p<\infty$, $F=H^p$ is strictly convex, for a given Finsler-Minkowski norm $H:\mathbb{R}^N\to[0,\infty)$ i.e. $H:\mathbb{R}^N\to[0,\infty)$ is $C^1(\mathbb{R}^N\setminus\{0\})$ and strictly convex such that $H(\xi)=0$ iff $\xi=0$, $H(t\xi)=|t|H(\xi)$ for every $\xi\in\R^N$, $t\in\R$ and $c_1|\xi|\leq H(\xi)\leq c_2|\xi|$ for some positive constants $c_1,c_2$ for all $\xi\in\mathbb{R}^N$. We refer the reader to \cite{Xiathesis, BFK, BCS} and the references therein for more details on $H$.
When $H(\xi)=|\xi|$, equation \eqref{waeqn} reduces to the following weighted singular $p$-Laplace equation
\begin{equation}\label{pmaineqn}
-\Delta_{p,w}u=f(x)u^{-\gamma(x)}\text{ in }\Om,\quad u>0\text{ in }\Om,\quad u=0\text{ on }\partial\Om,
\end{equation}
where 
$$
\Delta_{p,w}u:=\text{div}(w(x)|\nabla u|^{p-2}\nabla u),\quad 1<p<\infty
$$
is the weighted $p$-Laplace operator. We observe that $\Delta_{p,w}$ reduces to the $p$-Laplace operator $\Delta_p$ when $w\equiv 1$. Therefore, equation \eqref{maineqn} covers a wide range of singular problems.

Here singularity refers to the blow up property of the nonlinearity in the right hand side of \eqref{maineqn}, which occur due to the exponent $\gamma$. Elliptic equations with singular nonlinearities has been studied thoroughly over the last three decade in both the linear and nonlinear setting and there is a colossal amount of literature available in this direction, but most of them are restricted to bounded weight functions $w$ and with constant singular exponent $\gamma$. In this article, we consider the case when the weight function $w$ may vanish or blow up near the origin (for example, $w(x)=|x|^{\nu}$, $\nu\in\mathbb{R}$) and simultaneously $\gamma$ is allowed to vary inside the domain $\Omega$.

Let us discuss some known results for singular problems with bounded weight functions $w$ and constant singular exponent $\gamma$. For $p=2$, authors in \cite{CRT} established existence of a unique classical solution for the singular Laplace equation
\begin{equation}\label{lap}
-\Delta u=u^{-\gamma}\text{ in }\Om,\quad u>0\text{ in }\Om,\quad u=0\text{ on }\partial\Om,
\end{equation}
for any $\gamma>0$. Later it has been observed in \cite{LMckena} such a solution belong to $W_0^{1,2}(\Om)$ if and only if $\gamma<3$. Authors in \cite{Boc} removed this restriction of $\gamma$ to find weak solutions in $W^{1,2}_{\mathrm{loc}}(\Om)$ of the problem \eqref{lap} for any $\gamma>0$ by the approximation method. The associated singular $p$-Laplace equation of \eqref{lap} is studied in \cite{Canino,DeCave} and the references therein. In the perturbed case, for a certain range of $\la$ and $q$, the equation
\begin{equation}\label{purlap}
-\Delta u=\la u^{-\gamma}+u^{q}\text{ in }\Om,\quad u>0\text{ in }\Om,\quad u=0\text{ on }\partial\Om,
\end{equation}
possesses multiple weak solutions for any $0<\gamma<1$ as shown in \cite{arcoyaBoc,YHaitao,hirano1}, whereas for any $\gamma>0$, existence of one weak solution is proved in \cite{Boc1}, for multiplicity result in this concern, see \cite{Merida}. For multiplicity result of the $p$-Laplace analogue of \eqref{purlap}, see \cite{GST} for $0<\gamma<1$ and \cite{BGmed} for $\gamma\geq 1$. For study of singular measure data problems, we refer to \cite{Pettita} and the references therein.

When $\gamma$ is a variable, authors in \cite{CMP} provided sufficient conditions on $\gamma$ to establish existence of weak solutions for the problem \eqref{lap}. Such phenomenon has been extended by many authors to the variants of $p$-Laplace equations, see \cite{CGG,Zhang,Alves20,Mirivar, Alves18,BKcvpde,PSzamp,BGM, GMnonloc} and the references therein. We would like to point out that for singular problems, when $\gamma>1$, in general the solution $u$ does not belong to $W_{0}^{1,p}(\Om)$. This is compensated by a suitable power $\theta\geq 1$ so that $u^\theta\in W_0^{1,p}(\Om)$, which is referred to as the Dirichlet boundary condition $u=0$ on $\partial\Om$ (for example, see \cite{Boc, Canino}). 

It is worth mentioning that weighted singular problems are very less understood, when the weight function $w$ vanish or blow up near the origin even in the linear case $p=2$ in \eqref{pmaineqn}. In such situation, equation \eqref{pmaineqn} becomes degenerate that is captured by the weight function $w$, one can refer to \cite{Drabek,Juh,Gold0, Gold1} for a wide range of investigation of weighted problems with non-singular nonlinearities. Recently, authors in \cite{Garain, GM, BGaniso} studied the following type of weighted singular problems
\begin{equation}\label{bgmeqn}
-\Delta_{p,w}u=g(x,u)\text{ in }\Om,\quad u>0\text{ in }\Om,\quad u=0\text{ on }\partial\Om
\end{equation}
to deal with the question of existence for various type of singular nonlinearity $g$, where the weight $w$ belong to a class of Muckenhoupt weights. Recently, singular problems for more general $p$-admissible weights has been studied in \cite{GK, BG, Gtmna, Hara}.

In this article, we provide sufficient conditions on the weight function $w$, the variable singular exponent $\gamma$ and on the nonlinearity $f$ to ensure existence of weak solutions of \eqref{maineqn}. Further,  for a wider class of nonlinearity $f$, we establish existence and regularity results assuming the singular exponent $\gamma$ is a positive constant.

To prove our main results, we follow the approximation approach developed in \cite{Boc}, although we face several difficulties due to the weighted nonlinear structure. Some major difficulties in the weighted case are that suitable embedding results, regularity results are not readily available. We found a class of $p$-admissible weights $W_p^{s}$ defined in Section $2$ to be useful, for which the corresponding theory of the weighted Sobolev space is well-developed, see \cite{Juh, Drabek}. This class of weights allows us to shift from the weighted Sobolev space to a unweighted Sobolev space with a different Sobolev exponent (see Lemma \ref{emb}), which is useful to obtain several a priori estimates. Moreover, we observe that when $f=w$ our results hold for the whole class of $p$-addmissible weights $W_p$ to be defined in Section $2$. As in \cite{Hara},
we prove existence of the approximate solutions by a direct implementation of the Minty-Browder Theorem from \cite{var}. Moreover, due to the nonlinear structure of our operator, in contrast to \cite{Boc}, a priori estimates on the approximate solutions are not enough to pass to the
limit, see \cite{Canino}, which required a gradient convergence theorem in the unweighted setting, see \cite{Bocgrad, Masomura, TeroMaly}. Here we apply gradient convergence
theorem established for the weighted case in \cite{Mikko}, see also \cite{Tru}. Further, we use the technique from \cite{CMP} to deal with variable singularity.

This article is organized as follows: In Section $2$, we present the functional setting and state our main results. In Section $3$, we establish some preliminary results and finally, in Section $4$, we prove our main results.

\section{Functional setting and main results}
Throughout the rest of the article, we assume $1<p<\infty$, unless otherwise mentioned. We say that a function $w$ belong to the class of $p$-admissible weights $W_p$, if $w\in L^1_{\mathrm{loc}}(\mathbb{R}^N)$ such that $0<w<\infty$ almost everywhere in $\mathbb{R}^N$ and satisfy the following conditions:
\begin{enumerate}
\item[(i)] for any ball $B$ in $\mathbb{R}^N$, there exists a positive constant $C_{\mu}$ such that 
$$
\mu(2B)\leq C_{\mu}\,\mu(B),
$$
where 
$$
\mu(E)=\int_{E}w\,dx
$$
for a measurable subset $E$ in $\mathbb{R}^N$ and $d\mu(x)=w(x)\,dx$, where $dx$ is the $N$-dimensional Lebesgue measure.
\item[(ii)] If $D$ is an open set and $\{\phi_i\}_{i\in\mathbb{N}}\subset C^{\infty}(D)$ is a sequence of functions such that 
$$
\int_{D}|\phi_i|^p\,d\mu\to 0\text{ and } \int_{D}|\nabla\phi_i-v|^p\,d\mu\to 0
$$
as $i\to\infty$, where $v$ is a vector valued measurable function in $L^p(D,w)$, then $v=0$.
\item[(iii)] There exist constants $\kappa>1$ and $C_1>0$ such that 
\begin{equation}\label{wp}
\left(\frac{1}{\mu(B)}\int_{B}|\phi|^{\kappa p}\,d\mu\right)^\frac{1}{\kappa p}\leq C_1 r \left(\frac{1}{\mu(B)}\int_{B}|\nabla\phi|^p\,d\mu\right)^\frac{1}{p},
\end{equation}
whenever $B=B(x_0,r)$ is a ball in $\mathbb{R}^N$ centered at $x_0$ with radius $r$ and $\phi\in C_{c}^\infty(B)$.
\item[(iv)] There exists a constant $C_2>0$ such that
\begin{equation}\label{wp1}
\int_{B}|\phi-\phi_B|^p\,d\mu\leq C_2 r^p\int_{B}|\nabla\phi|^p\,d\mu,
\end{equation}
whenever $B=B(x_0,r)$ is a ball in $\mathbb{R}^N$ and $\phi\in C^\infty(B)$ is bounded. Here
$$
\phi_B=\frac{1}{\mu(B)}\int_{B}\phi\,d\mu.
$$
\end{enumerate}
The conditions (i)-(iv) are important in the theory of weighted Sobolev spaces, one can refer to \cite{Juh} for more details. \\
\textbf{Examples:} 
\begin{itemize}
\item Muckenhoupt weights $A_p$ are $p$-admissible, see \cite[Theorem 15.21]{Juh}.

\item Let $1<p<N$ and $J_f(x)$ denote the determinant of the Jacobian matrix of a $K$-quasiconformal mapping $f:\mathbb{R}^N\to\mathbb{R}^N$, then $w(x)=J_f(x)^{1-\frac{p}{N}}\in W_q$ for any $q\geq p$, see \cite[Corollary 15.34]{Juh}.

\item If $1<p<\infty$ and $\nu>-N$, then $w(x)=|x|^{\nu}\in W_p$, see \cite[Corollary 15.35]{Juh}.

\end{itemize}
For more examples, refer to \cite{ex1, ex2, ex3, Tero, Juh} and the references therein. 

\begin{Definition}(Weighted Spaces)
Let $1<p<\infty$ and $w\in W_p$. Then the weighted Lebesgue space $L^{p}(\Omega,w)$ is the class of measurable functions $u:\Om\to\mathbb{R}$ such that the norm of $u$ given by
\begin{equation}\label{lnorm}
\|u\|_{L^p(\Om,w)} = \Big(\int_{\Omega}|u(x)|^{p} w(x)\,dx\Big)^\frac{1}{p}<\infty.
\end{equation}
The weighted Sobolev space $W^{1,p}(\Om,w)$ is the class of measurable functions $u:\Om\to\mathbb{R}$ such that
\begin{equation}\label{norm1}
\|u\|_{1,p,w} = \Big(\int_{\Omega}|u(x)|^{p} w(x)\,dx+\int_{\Omega}|\nabla u(x)|^{p} w(x)\,dx\Big)^\frac{1}{p}<\infty.
\end{equation}
If $u\in W^{1,p}(\Om',w)$ for every $\Om'\Subset\Om$, then we say that $u\in W^{1,p}_{\mathrm{loc}}(\Om,w)$. The weighted Sobolev space with zero boundary value is defined as
$$
W^{1,p}_{0}(\Omega,w)=\overline{\big(C_{c}^{\infty}(\Omega),\|\cdot\|_{1,p,w}\big)}.
$$
\end{Definition}
Using the Poincar\'e inequality from \cite{Juh}, the norm defined by \eqref{norm1} on the space $W_{0}^{1,p}(\Omega,w)$ is equivalent to the norm given by
\begin{equation}\label{equinorm}
\|u\|_{W_0^{1,p}(\Om,w)}=\Big(\int_{\Omega}|\nabla u|^p w\,dx\Big)^\frac{1}{p}.
\end{equation}
Moreover, the space $W^{1,p}_{0}(\Omega,w)$ is a separable and uniformly convex Banach space, see \cite{Juh}.

We provide existence results for the following subclass of $p$-admissible weights given by $W_p^{s}$, see Theorem \ref{pur2thm1}. Moreover, we establish existence results for the whole class of $p$-admissible weights $w\in W_p$, in Theorem \ref{w2thm1} for $f=w$ in $\Om$. Let us remark that the class $W_{p}^{s}$ is introduced in \cite{Drabek} to study nonsingular weighted $p$-Laplace equations, where such class are crucial, for example to obtain boundedness estimates, which allows us to deal with the case $\gamma>1$ in Theorem \ref{pur2thm1}. To this end, below we state an embedding result. Let $1<p<\infty$ and define the set
\begin{equation}\label{I}
I=\Big[\frac{1}{p-1},\infty\Big)\cap\Big(\frac{N}{p},\infty\Big).
\end{equation}
Consider the following subclass of $W_p$ given by
\begin{equation}\label{wgtcls}
W_p^{s} = \Big\{w\in W_p: w^{-s}\in L^{1}(\Omega)\,\,\text{for some}\,\,s\in I\Big\}.
\end{equation}
Let $s\in I$, then we observe that
$$
w(x)=|x|^{\nu}\in W_p^{s}\text{ for any }\nu\in\Big(-N,\frac{N}{s}\Big).
$$
The following embedding result follows from \cite{Drabek}.
\begin{Lemma}\label{emb}
Let $1<p<\infty$ and $w \in W_p^{s}$ for some $s\in I$. Then the following continuous inclusion maps hold
\[
    W^{1,p}(\Omega,w)\hookrightarrow W^{1,p_s}(\Omega)\hookrightarrow 
\begin{cases}
    L^t(\Omega),& \text{for } p_s\leq t\leq p_s^{*}, \text{ if } 1\leq p_s<N, \\
    L^t(\Omega),& \text{ for } 1\leq t< \infty, \text{ if } p_s=N, \\
    C(\overline{\Omega}),& \text{ if } p_s>N,
\end{cases}
\]
where $p_s = \frac{ps}{s+1} \in [1,p)$.
Moreover, the second embedding above is compact except for $t=p_s^{*}=\frac{Np_s}{N-p_s}$, if $1\leq p_s<N$. Further, the same result holds for the space $W_{0}^{1,p}(\Omega,w)$.
\end{Lemma}

\begin{Remark}\label{Embrmk}
We note that if $0<c\leq w\leq d$ for some constants $c,d$, then $W^{1,p}(\Omega,w)=W^{1,p}(\Omega)$ and by the Sobolev embedding, Lemma \ref{emb} holds by replacing $p_s$ and $p_s^{*}$ with $p$ and $p^{*}=\frac{Np}{N-p}$ respectively.
\end{Remark}
Before stating our main results, let us define the notion of weak solutions in our setting.
\begin{Definition}\label{wksoldef}
Let $1<p<\infty$, $w\in W_p$, $f\in L^1(\Om)\setminus\{0\}$ is nonnegative and $\gamma\in C(\overline{\Omega})$ is positive. Then we say that $u\in W_{\mathrm{loc}}^{1,p}(\Om,w)$ is a weak solution of \eqref{maineqn} if $u>0$ in $\Om$ such that for every $\omega\Subset\Om$ there exists a positive constant $c(\omega)$ satisfying $u\geq c(\omega)>0$ in $\omega$ and for every $\phi\in C_c^{1}(\Om)$, one has
\begin{equation}\label{wksoleqn}
\int_{\Om}A(x,\nabla u)\nabla\phi\,dx=\int_{\Om}f(x)u^{-\gamma(x)}\phi\,dx,
\end{equation}
where $u=0$ on $\partial\Om$ is defined in the sense that for some $\theta\geq 1$, one has $u^{\theta}\in W_0^{1,p}(\Om,w)$.
\end{Definition}
Also, to deal with the variable exponent $\gamma$, we define a $\delta$- neighbourhood of $\Omega$ by
\begin{equation}\label{deltanbd}
\Omega_\delta:=\{x\in\Omega:\text{dist}(x,\partial\Omega)<\delta\},\quad \delta>0.
\end{equation}

\subsection{Statement of the main results:}
First, we state the following existence results concerning both variable and constant singular exponent. We remark that, for constant singular exponent $\gamma$, the assumptions on $f$ can be relaxed compared to the variable exponent as stated in the following result.
\begin{Theorem}\label{pur2thm1}(Existence)
Let $1<p<\infty$, $w\in W_p^{s}$ for some $s\in I$ and $\gamma\in C(\overline{\Omega})$ be positive.
\begin{enumerate}
\item[$(a)$] Suppose there exists $\delta>0$ such that $0<\gamma(x)\leq 1$ for every $x\in\Omega_\delta$, then the problem \eqref{maineqn} admits a weak solution in $W_0^{1,p}(\Omega,w)$, provided $f\in L^m(\Omega)\setminus\{0\}$ is nonnegative, where
\begin{equation}\label{vm1}
m =
\begin{cases}
(p_s^{*})^{'}, \text{ if } 1 \leq p_s < N, \text{ or }\\
>1\text{ if }p_s = N,\text{ or }\\
1\text{ if }p_s>N.
\end{cases}
\end{equation}

\item[$(b)$] Suppose $\gamma(x)\equiv \gamma\in(0,1)$ is a positive constant, then the problem \eqref{maineqn} admits a weak solution in $W_0^{1,p}(\Omega,w)$, provided $f\in L^m(\Omega)\setminus\{0\}$ is nonnegative, where
\begin{equation}\label{cm1}
m =
\begin{cases}
\big(\frac{p_s^{*}}{1-\gamma}\big)^{'}, \text{ if } 1 \leq p_s < N, \text{ or }\\
>1\text{ if }p_s = N,\text{ or }\\
1\text{ if }p_s>N.
\end{cases}
\end{equation}


\item[$(c)$] Suppose $\gamma(x)\equiv 1$, then for any nonnegative $f\in L^1(\Omega)\setminus\{0\}$, the problem \eqref{maineqn} admits a weak solution in $W_0^{1,p}(\Omega,w)$.

\item[$(d)$] Suppose there exists $\delta>0$ and $\gamma^*>1$ such that $\|\gamma\|_{L^\infty(\Omega_\delta)}\leq\gamma^*$, then the problem \eqref{maineqn} admits a weak solution $v\in W_{\mathrm{loc}}^{1,p}(\Omega,w)$ such that $v^\frac{\gamma^{*}+p-1}{p}\in W_0^{1,p}(\Omega,w)$, provided $f\in L^m(\Omega)\setminus\{0\}$ is nonnegative, where
\begin{equation}\label{vm2}
m =
\begin{cases}
\Big(\frac{(\gamma^{*}+p-1)p_s^{*}}{p\gamma^{*}}\Big)^{'}, \text{ if } 1 \leq p_s < N, \text{ or }\\
\Big(\frac{(\gamma^{*}+p-1)l}{p\gamma^{*}}\Big)^{'}\text{ if } \frac{p\gamma^{*}}{\gamma^{*}+p-1}<l<\infty\text{ and }p_s = N,\text{ or }\\
1\text{ if }p_s>N.
\end{cases}
\end{equation}

\item[$(e)$] Suppose $\gamma(x)\equiv \gamma>1$ is a positive constant, then for any nonnegative $f\in L^1(\Omega)\setminus\{0\}$, the problem \eqref{maineqn} admits a weak solution $v\in W_{\mathrm{loc}}^{1,p}(\Omega,w)$ such that $v^\frac{\gamma+p-1}{p}\in W_0^{1,p}(\Omega,w)$.
\end{enumerate}
\end{Theorem}

Moreover, we have the following existence result for any $w\in W_p$ when $f=w$ in $\Omega$.

\begin{Theorem}\label{w2thm1}(Existence)
Let $1<p<\infty$, $w\in W_p $ and $f=w$ in $\Om$. Assume that $\gamma\in C(\overline{\Omega})$ is positive. Suppose there exists $\delta>0$ such that $0<\gamma(x)\leq 1$ for every $x\in\Omega_\delta$, then the problem \eqref{maineqn} admits a weak solution in $W_0^{1,p}(\Omega,w)$.
\end{Theorem}
Further, we have the following regularity results when $\gamma(x)\equiv\gamma$ is a positive constant. Taking into account Lemma \ref{exisapprox}, the proof of Theorem \ref{regularity delta less}, \ref{regularity delta equal} and \ref{regularity delta greater} stated below follows along the lines of the proof of Theorem 3.3, 3.5 and 3.7 respectively in \cite{Garain}.
\begin{Theorem}\label{regularity delta less}(Regularity)
Let $\gamma(x)\equiv \gamma\in(0,1)$, then the solution $v$ given by Theorem \ref{pur2thm1}-$(b)$ satisfy the following properties:
\begin{enumerate}
\item[(a)] For $1 \leq p_s < N,$
\begin{enumerate}
\item[(i)] if $f \in L^{m}(\Omega)$ for some $m \in\big[(\frac{p_s^{*}}{1-\gamma})^{'},\frac{p_s^{*}}{p_s^{*}-p}\big)$, then $v \in L^t(\Omega)$, $t = p_s^{*}\,\delta$ where $\delta = \frac{(\gamma+p-1)m^{'}}{(pm^{'}-p_{s}^*)}$.
\item[(ii)] if $f \in L^m(\Omega)$ for some $m > \frac{p_s^{*}}{p_s^{*}-p}$, then $v \in L^\infty(\Omega)$.
\end{enumerate}
\item[(b)] Let $p_s = N$ and assume $q > p$. Then if $f \in L^m(\Omega)$ for some $m \in \big((\frac{q}{1-\delta})',\frac{q}{q-p}\big)$, we have $v \in L^t(\Omega)$, $t = p\,\delta$ where $\delta= \frac{pm'}{pm'-q}$.
\item[(c)] For $p_s > N$ and $f \in L^1(\Omega)$, we have $v \in L^\infty(\Omega)$.
\end{enumerate}
\end{Theorem}
\begin{Theorem}\label{regularity delta equal}(Regularity)
Let $\gamma(x)\equiv 1$, then the solution $v$ given by Theorem \ref{pur2thm1}-$(c)$ satisfy the following properties:
\begin{enumerate}
\item[(a)] For $1 \leq p_s < N$,
\begin{enumerate}
\item[(i)] if $f  \in L^{m}(\Omega)$ for some $m \in \big(1,\frac{p_s^{*}}{p_s^{*}-p}\big)$, then $v \in L^t(\Omega)$, $t = p_s^{*}\delta$, where $\delta = \frac{pm^{'}}{(pm^{'}-p_{s}^*)}$.
\item[(ii)] if $f \in L^m(\Omega)$ for some $m > \frac{p_s^{*}}{p_s^{*}-p}$, then $v \in L^\infty(\Omega)$.
\end{enumerate}
\item[(b)] Let $p_s = N$ and $q > p$. Then if $f \in L^m(\Omega)$ for some $m \in\big(1,\frac{q}{q-p}\big)$, we have $v \in L^t(\Omega)$, $t = q\,\gamma$, where $\gamma = \frac{pm^{'}}{pm^{'}-q}$.
\item[(c)] For $p_s > N$ and $f \in L^1(\Omega)$, we have $v \in L^\infty(\Omega)$.
\end{enumerate}
\end{Theorem}
\begin{Theorem}\label{regularity delta greater}(Regularity)
Let $\gamma(x)\equiv \gamma>1,$ then the solution $v$ given by Theorem \ref{pur2thm1}-$(e)$ satisfies the following properties:
\begin{enumerate}
\item[(a)] For $1 \leq p_s < N,$
\begin{enumerate}
\item[(i)] if $f \in L^{m}(\Omega)$ for some $m \in \big(1,\frac{p_s^{*}}{p_s^{*}-p}\big)$, then $v \in L^t(\Omega)$ where $t = p_s^{*}\,\delta$, where $\delta = \frac{(\gamma+p-1)m'}{pm'-p_s^*}$.
\item[(ii)] if $f \in L^m(\Omega)$ some $m > \frac{p_s^{*}}{p_s^{*}-p}$, then $v \in L^\infty(\Omega)$.
\end{enumerate}
\item[(b)] Let $p_s = N$ and assume $q > p$. Then if $f \in L^m(\Omega)$ for some $m \in \big(1,\frac{q}{q-p}\big)$,  we have $v \in L^t(\Omega)$, $t = q\,\delta$, where $\delta = \frac{(\gamma+p-1)m^{'}}{pm^{'}-q}$.
\item[(c)] For $p_s > N$ and $f \in L^1(\Omega)$, we have $v \in L^\infty(\Omega)$.
\end{enumerate}
\end{Theorem}
\begin{Remark}\label{prmk1}
If $c\leq w\leq d$ for some positive constant $c,d$, then taking into account Remark \ref{Embrmk} and arguing similarly, it follows that all the above existence and regularity results will be valid by replacing $p_s$ and $p_s^{*}$ with $p$ and $p^{*}$ respectively.
\end{Remark}
\textbf{Notation:} For $u\in W_0^{1,p}(\Om,w)$, denote by $\|u\|$ to mean the norm $\|u\|_{W_0^{1,p}(\Om,w)}$ as defined by \eqref{equinorm}. For given constants $c,d$, a set $S$ and a function $u$, by $c\leq u\leq d$ in $S$, we mean $c\leq u\leq d$ almost everywhere in $S$. Moreover, we write $|S|$ to denote the Lebesgue measure of $S$. The symbol $\langle,\rangle$ denotes the standard inner product in $\mathbb{R}^N$. The conjugate exponent of $\theta>1$ is denoted by $\theta'=\frac{\theta}{\theta-1}$. We denote by $p^*=\frac{Np}{N-p}$ if $1<p<N$ and $p_{s}^*=\frac{Np_{s}}{N-p_{s}}$ if $1\leq p_s<N$. For $a\in\mathbb{R}$, we denote by $a^+:=\max\{a,0\}$, $a^-=\max\{-a,0\}$. We write by $c,C$ or $c_i,C_i$ for $i\in\mathbb{N}$ to mean a constant which may vary from line to line or even in the same line. The dependencies of the constants are written in parenthesis.

\section{Preliminary results}
In the spirit of Boccardo-Orsina \cite{Boc}, for $1<p<\infty$, we investigate the following approximate problem
\begin{equation}\label{approx}
-\text{div}(A(x,\nabla u))=f_n(x)\Big(u^{+}+\frac{1}{n}\Big)^{-\gamma(x)}\text{ in }\Omega,\quad u=0\text{ on }\partial\Omega,
\end{equation}
where $f_n(x)=\min\{f(x),n\}$ for $x\in\Omega$, $n\in\mathbb{N}$, $f\in L^1(\Omega)\setminus\{0\}$ is a nonnegative function, $\gamma\in C(\overline{\Omega})$ is positive and $w\in W_p$. We obtain existence of solutions for the problem \eqref{approx} for each $n\in\mathbb{N}$ and prove some a priori estimates below to prove our main results. To this end, first we state the following result from \cite[Theorem $9.14$]{var}, which is useful to prove Lemma \ref{exisapprox} below.

\begin{Theorem}\label{MB}
Let $V$ be a real separable reflexive Banach space and $V^*$ be the dual of $V$. Suppose that $T:V\to V^{*}$ is a coercive and demicontinuous monotone operator. Then $T$ is surjective, i.e., given any $f\in V^{*}$, there exists $u\in V$ such that $T(u)=f$. If $T$ is strictly monotone, then $T$ is also injective.  
\end{Theorem}

\begin{Lemma}\label{exisapprox}
Let $1<p<\infty$. Suppose $w\in W_p^{s}$ for some $s\in I$ or $w=f$ in $\Om$. Then for every $n\in\mathbb{N}$, the problem \eqref{approx} admits a positive weak solution $v_{n}\in W_0^{1,p}(\Om,w)$ such that for every $\omega\Subset\Omega$, there exists a positive constant $C(\omega)$ satisfying $v_n\geq C(\omega)>0$ in $\omega$.
Moreover, $v_n$ is unique for every $n\in\mathbb{N}$ and $v_{n+1}\geq v_n$ for every $n\in\mathbb{N}$ in $\Omega$. In addition, $\{v_n\}_{n\in\N}\subset L^\infty(\Om)$, if $w\in W_p^{s}$ for some $s\in I$.
\end{Lemma}
\begin{proof}
We prove the result only by considering the case when $w\in W_p^{s}$ for some $s\in I$, since if $w=f$ in $\Om$, the proof is similar. First, we prove the existence of $v_n$. Let $V=W_0^{1,p}(\Om,w)$ and $V^*$ be the dual of $V$. We define $T:V\to V^*$ by
$$
\langle T(v),\phi\rangle=\int_{\Om}A(x,\nabla v)\nabla\phi\,dx-\int_{\Om}f_n(x)\Big(v^{+}+\frac{1}{n}\Big)^{-\gamma(x)}\phi\,dx,\quad \forall v,\phi\in V.
$$
Note that $T$ is well defined. Indeed, using the H$\ddot{\text{o}}$lder's inequality and hypothesis $(H_1)$ along Lemma \ref{emb}, we obtain 
\begin{align*}
|\langle T(v),\phi\rangle\big|&=\left|\int_{\Om}A(x,\nabla v)\nabla\phi\,dx-\int_{\Om}f_n(x)\Big(v^{+}+\frac{1}{n}\Big)^{-\gamma(x)}\phi\,dx\right|\\
&\leq \beta\int_{\Om}(w^\frac{1}{p^{'}}|\nabla v|^{p-1})(w^\frac{1}{p}|\nabla\phi|)\,dx+n^{1+\|\gamma\|_{L^\infty(\Om)}}\int_{\Om}|\phi|\,dx\\
&\leq C(\|v\|^{p-1}\|\phi\|+\|\phi\|)\leq C\|\phi\|,
\end{align*}
for some positive constant $C$. \\
\textbf{Coercivity:} Using $(H_2)$, H\"older's inequality and Lemma \ref{emb}, we obtain
$$
\langle T(v),v\rangle=\int_{\Om}A(x,\nabla v)\nabla v\,dx-\int_{\Om}f_n(x)\Big(v^+ +\frac{1}{n}\Big)^{-\gamma(x)}v\,dx\geq \alpha\|v\|^{p}-C\|v\|.
$$
Therefore, since $1<p<\infty$, $T$ is coercive.\\
\textbf{Demicontinuity:} Let $v_k\to v$ in the norm of $V$ as $k\to\infty$. Then $w^\frac{1}{p}\nabla v_k\to w^\frac{1}{p}\nabla v$ strongly in $L^p(\Omega)$ as $k\to\infty$. Therefore up to a subsequence, still denoted by $v_k$, we have $v_k(x)\to v(x)$ and $\nabla v_{k}(x)\to\nabla v(x)$ pointwise for a.e. $x\in\Omega$. Since the function $A(x,\cdot)$ is continuous in the second variable, we have
$$
w(x)^{{-}\frac{1}{p}}A\big(x,\nabla v_{k}(x)\big)\to w(x)^{-\frac{1}{p}}A\big(x,\nabla v(x)\big)
$$ pointwise for a.e. $x\in\Omega.$ Now using the growth condition $(H_1)$ and the norm boundedness of $v_k$, we obtain
\begin{align*}
||w^{-\frac{1}{p}}A(x,\nabla v_{k})||^{p^{'}}_{L^{p^{'}}(\Omega)}
&=\int_{\Omega}w^{-\frac{1}{p-1}}(x)\big|A\big(x,\nabla v_{k}(x)\big)\big|^{p^{'}}\,dx\\
&\leq\beta^{p^{'}}\int_{\Omega}w(x)|\nabla v_{k}(x)|^p\,dx\leq C,
\end{align*}
for some positive constant $C$, independent of $k$. Therefore, upto a subsequence
$$
w^{-\frac{1}{p}}A\big(x,\nabla v_{k}(x)\big)\rightharpoonup w^{-\frac{1}{p}}A\big(x,\nabla v(x)\big)
$$ weakly in $L^{p^{'}}(\Om)$ and since the weak limit is independent of the choice of the subsequence, the above weak convergence holds for every $k$. Now $\phi\in V$ implies that $w^\frac{1}{p}\nabla\phi\in L^p(\Omega)$ and therefore by the above weak convergence, we obtain
\begin{equation}\label{d1}
\lim_{k\to\infty}\int_{\Om}A(x,\nabla v_k(x))\nabla\phi\,dx=\int_{\Om}A(x,\nabla v(x))\nabla\phi\,dx.
\end{equation}
Moreover, since $v_k(x)\to v(x)$ pointwise a.e. in $\Om$ and for every $\phi\in V$,  
$$
\left|f_n(v_k^{+}+\frac{1}{n})^{-\gamma(x)}\phi-f_n(v^{+}+\frac{1}{n})^{-\gamma(x)}\phi\right|\leq n^{1+\|\gamma\|_{L^\infty(\Om)}}|\phi|.
$$
Note that by Lemma \ref{emb}, $\phi\in L^1(\Om)$. Then using the Lebesgue's dominated convergence theorem, we have
\begin{equation}\label{d2}
\lim_{k\to\infty}\int_{\Om}f_n\Big(v_k^{+}+\frac{1}{n}\Big)^{-\gamma(x)}\phi\,dx=\int_{\Om}f_n\Big(v^{+}+\frac{1}{n}\Big)^{-\gamma(x)}\phi\,dx,\quad \forall\phi\in V.
\end{equation}
Therefore, from \eqref{d1} and \eqref{d2}, it follows that
$$
\lim_{k\to\infty}\langle T(v_k),\phi\rangle=\langle T(v),\phi\rangle,\quad \forall\phi\in V
$$
and hence $T$ is demicontinuous.\\
\textbf{Monotonicity of $T$:} Let $v_1,v_2\in V$, then we observe that
\begin{equation}\label{tmn}
\begin{split}
\langle T(v_1)-T(v_2),v_1 -v_2\rangle&=\int_{\Om}\langle A(x,\nabla v_1)-A(x,\nabla v_2),\nabla(v_1- v_2)\rangle\,dx\\
&\quad-\int_{\Om}f_n(x)\left\{\Big(v_1^{+}+\frac{1}{n}\Big)^{-\gamma(x)}-\Big(v_2^{+}+\frac{1}{n}\Big)^{-\gamma(x)}\right\}(v_1- v_2)\,dx
\end{split}
\end{equation}
Note that by $(H_3)$, the first integral above is nonnegative. Moreover, it is easy to see that the second integral above is nonpositive. Therefore, $T$ is monotone.

Thus, by Theorem \ref{MB}, $T$ is surjective. Hence for every $n\in\mathbb{N}$, there exists $v_n\in V$ such that  we have
\begin{equation}\label{vneqn}
\int_{\Om}A(x,\nabla v_n)\nabla\phi\,dx=\int_{\Om}f_n(x)\Big(v_n^{+} +\frac{1}{n}\Big)^{-\gamma(x)}\,\phi\,dx,\quad\forall \phi\in V.
\end{equation}
Next, we prove the following properties of $v_n$.\\
\textbf{Boundedness:} Proceeding along the lines of the proof of \cite[Page 13, Theorem $3.13$]{Garain}, it follows that $v_n\in L^\infty(\Om)$ for every $n\in\mathbb{N}$.\\
\textbf{Positivity:} Choosing $\phi=\text{min}\big\{v_n,0\big\}$ as a test function in the equation \eqref{vneqn} we get $v_n\geq 0$ in $\Omega$. Since $f\neq 0$ in $\Om$, by \cite[Theorem $3.59$]{Juh} for every $\omega\Subset\Om$, there exists a positive constant $C=C(\omega,n)$ such that 
\begin{equation}\label{vnpr}
v_n\geq C(\omega,n)>0\text{ in } \omega.
\end{equation}
\textbf{Monotonicity of $v_n$:} From the above step, recalling that $v_n\geq 0$ in $\Om$ and choosing $\phi = (v_{n}-v_{n+1})^+$ as a test function in \eqref{vneqn}, we have
\begin{equation}\label{monopf}
\begin{split}
&\int_{\Omega}\langle A(x,\nabla v_n)-A(x,\nabla v_{n+1}),\nabla(v_n-v_{n+1})^{+}\rangle\,dx\\
&=\int_{\Omega}\left\{f_{n}\Big(v_{n}+\frac{1}{n}\Big)^{-\gamma(x)}-f_{n+1}\Big(v_{n+1}+\frac{1}{n+1}\Big)^{-\gamma(x)}\right\}(v_n-v_{n+1})^{+}\,dx\\
&\leq\int_{\Omega}f_{n+1}\left\{\Big(v_{n}+\frac{1}{n}\Big)^{-\gamma(x)}-\Big(v_{n+1}+\frac{1}{n+1}\Big)^{-\gamma(x)}\right\}(v_n-v_{n+1})^{+}\,dx\leq 0,
\end{split}
\end{equation}
where the last inequality above follows using $f_{n}(x) \leq f_{n+1}(x)$ and the positivity of $\gamma$. Therefore by $(H_3)$ we obtain $\nabla v_n=\nabla v_{n+1}$ in $\{x\in\Om:v_n(x)>v_{n+1}(x)\}$. Hence, we have $v_{n+1}\geq v_n$ in $\Om$. In particular, we get $v_n\geq v_1$ in $\Om$ and thus, by \eqref{vnpr} for every $\omega\Subset\Om$, there exists a positive constant $C(\omega)$ (independent of $n$) such that $v_n\geq C(\omega)>0$ in $\omega$. Uniqueness of $v_n$ follows similarly as monotonicity. This completes the proof. 
\end{proof}

Next, we obtain some a priori estimates for $v_n$. Before stating them, we define
\begin{equation}\label{deltanbd2}
\omega_\delta=\Omega\setminus\overline{\Omega_\delta},\quad\delta>0,
\end{equation}
where $\Omega_{\delta}=\{x\in\Om:\text{dist}(x,\partial\Om)<\delta\}$ is defined in \eqref{deltanbd}.
\begin{Lemma}\label{apun}
Let $1<p<\infty$ and $w\in W_p^{s}$ for some $s\in I$. Assume that $\{v_n\}_{n\in\mathbb{N}}$ is the sequence of solutions of \eqref{approx} given by Lemma \ref{exisapprox}.
\begin{enumerate}
\item[$(a)$] Suppose there exists $\delta>0$ such that $0<\gamma(x)\leq 1$ for every $x\in\Omega_\delta$, then $\{v_n\}_{n\in\mathbb{N}}$ is uniformly bounded in $W_0^{1,p}(\Omega,w)$, provided $f\in L^m(\Omega)\setminus\{0\}$ is nonnegative, where $m$ is given in \eqref{vm1}.
\item[$(b)$] If $\gamma(x)\equiv \gamma\in(0,1)$ is a positive constant, then $\{v_n\}_{n\in\mathbb{N}}$ is uniformly bounded in $W_0^{1,p}(\Om,w)$, provided $f\in L^m(\Omega)\setminus\{0\}$ is nonnegative, where $m$ is given in \eqref{cm1}. 
\item[$(c)$] If $\gamma(x)\equiv 1$, then $\{v_n\}_{n\in\mathbb{N}}$ is uniformly bounded in $W_0^{1,p}(\Om,w)$ for any nonnegative $f\in L^1(\Om)\setminus\{0\}$.
\item[$(d)$] Suppose there exists $\delta>0$ and $\gamma^*>1$ such that $\|\gamma\|_{L^\infty(\Omega_\delta)}\leq\gamma^*$, then the sequences $\{v_n\}_{n\in\mathbb{N}}$ and $\left\{v_n^{\frac{\gamma^{*}+p-1}{p}}\right\}_{n\in\mathbb{N}}$ are uniformly bounded in $W_{\mathrm{loc}}^{1,p}(\Omega,w)$ and $W_0^{1,p}(\Omega,w)$ respectively, provided $f\in L^m(\Omega)\setminus\{0\}$ is nonnegative, where $m$ is given in \eqref{vm2}.

\item[$(e)$] If $\gamma(x)=\gamma>1$ is a constant, then the sequences $\{v_n\}_{n\in\mathbb{N}}$ and $\left\{v_n^{\frac{\gamma^{*}+p-1}{p}}\right\}_{n\in\mathbb{N}}$ are uniformly bounded in $W_{\mathrm{loc}}^{1,p}(\Omega,w)$ and $W_0^{1,p}(\Omega,w)$ respectively, for any nonnegative $f\in L^1(\Omega)\setminus\{0\}$.
\end{enumerate}
\end{Lemma}
\begin{proof}
We prove the result only for $1\leq p_s<N$, since the other cases are analogous.
\begin{enumerate}
\item[$(a)$] Choosing $\phi=v_n$ as a test function in the weak formulation of \eqref{approx} and using $(H_2)$, we get
\begin{equation}\label{api}
\begin{split}
\alpha\int_{\Omega}w(x)|\nabla v_n|^p\,dx&\leq\int_{\Omega}f_n\Big(v_n+\frac{1}{n}\Big)^{-\gamma(x)}v_n\,dx\leq\int_{\overline{\Omega_\delta}\cap\{0<v_n\leq 1\}}fv_n^{1-\gamma(x)}\,dx\\&+\int_{\overline{\Omega_\delta}\cap\{v_n>1\}}fv_n^{1-\gamma(x)}\,dx
+\int_{\omega_\delta}f\|c(\omega_\delta)^{-\gamma(x)}\|_{L^\infty(\Omega)}v_n\,dx\\
&\qquad\leq\|f\|_{L^1(\Omega)}+(1+\|c(\omega_\delta)^{-\gamma(x)}\|_{L^\infty(\Omega)})\int_{\Omega}fv_n\,dx,
\end{split}
\end{equation}
where we have used the fact that $0<\gamma(x)\leq 1$ for every $x\in\Omega_\delta$ and the property $v_n\geq c(\omega_\delta)>0$ in $\omega_\delta$ from Lemma \ref{exisapprox}. Since $f\in L^m(\Omega)$ for $m=(p_s^{*})'$, using H\"older's inequality and Lemma \ref{emb} in \eqref{api}, we get
\begin{equation*}
\begin{split}
\alpha\|v_n\|^p&\leq \|f\|_{L^1(\Omega)}+(1+\|c(\omega_\delta)^{-\gamma(x)}\|_{L^\infty(\Omega)})\|f\|_{L^m(\Omega)}\|v_n\|_{L^{p_s^{*}}(\Omega)}\\
&\qquad\leq \|f\|_{L^1(\Omega)}+C\|f\|_{L^m(\Omega)}\|v_n\|,
\end{split}
\end{equation*}
for some positive constant $C$, independent of $n$. Therefore, the sequence $\{v_n\}_{n\in\mathbb{N}}$ is uniformly bounded in $W_0^{1,p}(\Omega,w)$.

\item[$(b)$] Let $\gamma\in(0,1)$. Choosing $\phi=v_n$ as a test function in \eqref{approx}, using Lemma \ref{emb}, we have
\begin{equation}\label{api1}
\begin{split}
\alpha\| v_n\|^p=\int_{\Omega}f_n\Big(v_n+\frac{1}{n}\Big)^{-\gamma}v_n\,dx\leq\int_{\Om}fv_n^{1-\gamma}\,dx\leq\|f\|_{L^m(\Omega)}\|v_n\|^{1-\gamma},
\end{split}
\end{equation}
where $m=(\frac{p_s^{*}}{1-\gamma})^{'}$. 
\item[$(c)$] If $\gamma=1$, then again choosing $\phi=v_n$ in \eqref{approx} and proceeding along the lines of \eqref{api1}, we get $\|v_n\|^{p}_{W_0^{1,p}(\Om,w)}\leq\frac{1}{\alpha}\|f\|_{L^1(\Om)}$.
\item[$(c)$] Note that by Lemma \ref{exisapprox}, we know that $v_n\in W_0^{1,p}(\Omega,w)\cap L^\infty(\Omega)$. Therefore, choosing $v_n^{\gamma^*}$ as a test function in the weak formulation of \eqref{approx} and using $(H_2)$, we obtain
\begin{equation}\label{apigrt1}
\begin{split}
&\alpha\gamma^*\Big(\frac{p}{\gamma^{*}+p-1}\Big)^p\int_{\Omega}w(x)\Big|\nabla v_n^\frac{\gamma^{*}+p-1}{p}\Big|^{p}\,dx\leq\int_{\Omega}f_n\Big(v_n+\frac{1}{n}\Big)^{-\gamma(x)}v_n^{\gamma^*}\,dx\\
&\qquad\leq\int_{\overline{\Omega_\delta}\cap\{0<v_n\leq 1\}}fv_n^{\gamma^{*}-\gamma(x)}\,dx+\int_{\overline{\Omega_\delta}\cap\{v_n>1\}}fv_n^{\gamma^{*}-\gamma(x)}\,dx+
\int_{\omega_\delta}f\|c(\omega_\delta)^{-\gamma(x)}\|_{L^\infty(\Omega)}v_n^{\gamma^*}\,dx\\
&\qquad\leq\|f\|_{L^1(\Omega)}+(1+\|c(\omega_\delta)^{-\gamma(x)}\|_{L^\infty(\Omega)})\int_{\Omega}fv_n^{\gamma^*}\,dx\\
&\qquad\leq \|f\|_{L^1(\Omega)}+(1+\|c(\omega_\delta)^{-\gamma(x)}\|_{L^\infty(\Omega)})\int_{\Omega}f\Big(v_n^{\frac{\gamma^{*}+p-1}{p}}\Big)^\frac{p\gamma^*}{\gamma^*+p-1}\,dx\\
&\qquad\leq\|f\|_{L^1(\Omega)}+c
\|f\|_{L^m(\Omega)}\left\|v_n^{\frac{\gamma^{*}+p-1}{p}}\right\|^\frac{l}{m^{'}},
\end{split}
\end{equation}
for some positive constant $c$, where we have used Lemma \ref{emb}, the hypothesis $\|\gamma\|_{L^\infty(\Omega_\delta)}\leq\gamma^*$ along with $v_n\geq c(\omega)>0$ from Lemma \ref{exisapprox} and also the fact that $v_{n}^\frac{\gamma^*+p-1}{p}\in W_0^{1,p}(\Omega,w)$, which is true since $v_n\in W_0^{1,p}(\Omega,w)\cap L^\infty(\Omega)$. Since $p>\frac{l}{m^{'}}$, it follows from \eqref{apigrt1} that $\left\{v_{n}^\frac{\gamma^{*}+p-1}{p}\right\}_{n\in\mathbb{N}}$ is uniformly bounded in $W_0^{1,p}(\Omega,w)$. Using this fact and H\"older's inequality, for any $\omega\Subset\Omega$, we observe that 
\begin{equation}\label{lpbd}
\int_{\omega}w v_{n}^p\,dx\leq\left(\int_{\omega}w v_{n}^{\gamma^{*}+p-1}\,dx\right)^\frac{p}{\gamma^{*}+p-1}\|w\|_{L^1(\omega)}^\frac{\gamma^{*}-1}{\gamma^{*}+p-1}\leq C, 
\end{equation}
for some positive constant $C$ independent of $n$ and
\begin{equation}\label{gradbd}
\begin{split}
\int_{\omega}w\Big|\nabla v_{n}^\frac{\gamma^{*}+p-1}{p}\Big|^p\,dx&=\Big(\frac{\gamma^{*}+p-1}{p}\Big)^p\int_{\omega}wv_{n}^{\gamma^{*}-1}|\nabla v_n|^p\,dx\\
&\qquad\geq c(\omega)^{\gamma^*-1}\Big(\frac{\gamma^{*}+p-1}{p}\Big)^p\int_{\omega}w|\nabla v_n|^p\,dx.
\end{split}
\end{equation}
In the last line above, we have again used the fact that $v_n\geq c(\omega)>0$ for every $\omega\Subset\Omega$ for some positive constant $c(\omega)$, independent of $n$. Therefore, the estimates \eqref{lpbd} and \eqref{gradbd} yields the uniform boundedness of $\{v_n\}_{n\in\mathbb{N}}$ in $W^{1,p}_{\mathrm{loc}}(\Omega,w)$. This completes the proof.

\item[$(e)$] Let $\gamma(x)\equiv \gamma>1$. Again, note that by Lemma \ref{approx}, we know that $v_n\in W_0^{1,p}(\Omega,w)\cap L^\infty(\Omega)$ and as in $(d)$ above, choosing $v_n^{\gamma}$ as a test function in \eqref{approx} and using $(H_3)$, we obtain
\begin{equation}\label{apigrt}
\begin{split}
&\alpha\gamma\Big(\frac{p}{\gamma+p-1}\Big)^p\int_{\Omega}w(x)\Big|\nabla v_n^\frac{\gamma+p-1}{p}\Big|^{p}\,dx=\int_{\Omega}f_n\Big(v_n+\frac{1}{n}\Big)^{-\gamma}v_n^{\gamma}\,dx\leq\|f\|_{L^1(\Om)}.
\end{split}
\end{equation}
Thus, proceeding similarly as in $(c)$ above, the result follows.
\end{enumerate}
\end{proof}

\begin{Lemma}\label{est2}
Let $1<p<\infty$, $w\in W_p$ and $f=w$ in $\Om$. Suppose $\{v_n\}_{n\in\mathbb{N}}$ is the sequence of solutions of \eqref{approx} given by Lemma \ref{exisapprox}. If there exists $\delta>0$ such that $0<\gamma(x)\leq 1$ for every $x\in\Omega_\delta$, then $\{v_n\}_{n\in\mathbb{N}}$ is uniformly bounded in $W_0^{1,p}(\Omega,w)$.
\end{Lemma}
\begin{proof}
Choosing $\phi=v_n$ as a test function in the weak formulation of \eqref{approx} and proceeding along the lines of the proof of $(a)$ in Lemma \ref{apun}, we have
\begin{equation}\label{a1}
\alpha\|v_n\|^p\leq\|f\|_{L^1(\Om)}+(1+\|c(\omega_{\delta}))^{-\gamma(x)}\|_{L^\infty(\Om)})\int_{\Om}fv_n\,dx\leq\|w\|_{L^1(\Om)}+c\|w\|_{L^1(\Om)}^\frac{1}{p^{'}}\|v_n\|,
\end{equation}
for some positive constant $c$. Since $1<p<\infty$, the result follows.
\end{proof}
\section{Proof of the main results}
\textbf{Proof of Theorem \ref{pur2thm1}:}
\begin{enumerate}
\item[$(a)$] Let $1\leq p_s<N$ and $f\in L^m(\Om)\setminus\{0\}$ be nonnegative, where $m=(p_s^{*})^{'}$. Then, by Lemma \ref{exisapprox}, for every $n\in\mathbb{N}$, there exists $v_n\in W_0^{1,p}(\Omega,w)$ such that
\begin{equation}\label{t1eqn1}
\int_{\Omega}A(x,\nabla v_n)\nabla\phi\,dx=\int_{\Omega}f_n\Big(v_n+\frac{1}{n}\Big)^{-\gamma(x)}\phi\,dx,\quad\forall\phi\in C_c^{1}(\Omega).
\end{equation}
By Lemma \ref{apun}-$(a)$, we have $\|v_n\|\leq C$, for some positive constant $C$, independent of $n$. Using this fact and taking into account the monotonicity property $v_{n+1}\geq v_n$ for every $n\in\mathbb{N}$ from Lemma \ref{exisapprox}, there exists $v\in W_0^{1,p}(\Omega,w)$ such that $0<v_n\leq v$ a.e. in $\Om$ and up to a susequence $v_n\to v$ pointwise a.e. in $\Omega$ and weakly in $W_0^{1,p}(\Omega,w)$. Moreover, by Lemma \ref{exisapprox}, for every $\omega\Subset\Omega$, there exists a positive constant $c(\omega)$, independent of $n$ such that $v\geq v_n\geq c(\omega)>0$ in $\omega$, which gives
\begin{equation}\label{rhs}
\left|f_n\Big(v_n+\frac{1}{n}\Big)^{-\gamma(x)}\phi\right|\leq \|c_{\mathrm{supp}\,\phi}^{-\gamma(x)}\,\phi\|_{L^\infty(\Omega)}f\in L^1(\Omega),
\end{equation}
for every $\phi\in C_c^{1}(\Omega)$.
Note that since $0<v_n\leq v$ a.e. in $\Om$, by the Lebesgue's dominated convergence theorem, $v_n\to v$ strongly in $L^p(\Om,w)$. Thus, by \cite[Theorem 2.16]{Mikko}, up to a subsequence $\nabla v_n\to \nabla v$ pointwise a.e. in $\Omega$ and hence, we obtain
\begin{equation}\label{lhs}
\lim_{n\to\infty}\int_{\Omega}A(x,\nabla v_n)\nabla\phi\,dx=\int_{\Omega}A(x,\nabla v)\nabla\phi\,dx,
\end{equation}
for every $\phi\in C_c^{1}(\Omega)$. On the other hand, by \eqref{rhs} and the Lebsegue's dominated convergence theorem, we obtain
\begin{equation}\label{rhslim}
\lim_{n\to\infty}\int_{\Omega}f_n\Big(v_n+\frac{1}{n}\Big)^{-\gamma(x)}\phi\,dx=\int_{\Omega}fv^{-\gamma(x)}\phi\,dx,
\end{equation}
for every $\phi\in C_c^{1}(\Omega)$. Thus using \eqref{lhs} and \eqref{rhslim} in \eqref{t1eqn1}, the result follows. The proof for $p_s\geq N$ is analogous.

\item[$(b)$] By Lemma \ref{apun}-$(b)$, we have $\{v_n\}_{n\in\mathbb{N}}$ is uniformly bounded in $W_0^{1,p}(\Om,w)$. Now, proceeding along the lines of the proof of part $(a)$ above, the result follows.

\item[$(c)$]  By Lemma \ref{apun}-$(c)$, we have $\{v_n\}_{n\in\mathbb{N}}$ is uniformly bounded in $W_0^{1,p}(\Om,w)$. Now, proceeding along the lines of the proof of part $(a)$ above, the result follows.

\item[$(d)$] By Lemma \ref{apun}-$(d)$, the sequences $\{v_n\}_{n\in\mathbb{N}}$ and $\left\{v_n^{\frac{\gamma^{*}+p-1}{p}}\right\}_{n\in\mathbb{N}}$ are uniformly bounded in $W^{1,p}_{\mathrm{loc}}(\Omega,w)$ and $W^{1,p}_{0}(\Omega,w)$ respectively. Then, analogous to the proof of $(a)$ above, the result follows.

\item[$(e)$] By Lemma \ref{apun}-$(e)$, the sequences $\{v_n\}_{n\in\mathbb{N}}$ and $\left\{v_n^{\frac{\gamma+p-1}{p}}\right\}_{n\in\mathbb{N}}$ are uniformly bounded in $W^{1,p}_{\mathrm{loc}}(\Omega,w)$ and $W^{1,p}_{0}(\Omega,w)$ respectively. Then, analogous to the proof of $(a)$ above, the result follows.  
\qed
\end{enumerate}

\textbf{Proof of Theorem \ref{w2thm1}:} If $f=w$ in $\Om$, by Lemma \ref{est2}, the sequence $\{v_n\}_{n\in\mathbb{N}}$ is uniformly bounded in $W_0^{1,p}(\Om,w)$. Then, following the exact proof of $(a)$ above in Theorem \ref{pur2thm1}, the result follows. \qed



\end{document}